\documentclass[a4paper,12pt]{amsart}
\usepackage[cp1251]{inputenc}
\usepackage[english, ukrainian]{babel}
\usepackage[all,dvipdfm]{xy}
\usepackage{amssymb,upref}
\usepackage[dvips]{graphicx}
\graphicspath{{C:\Users\Aspirant\Desktop\Новая папка}}
\usepackage{setspace}
\usepackage{xcolor}
\usepackage{fancyhdr}
\usepackage{graphicx}

\theoremstyle{plain}
\newtheorem{theorem}{Theorem}
\newtheorem{lemma}{Lemma}
\newtheorem{corollary}{Corollary}
\newtheorem{property}{Property}
\newtheorem{prop}{Proposition}

\theoremstyle{definition}
\newtheorem{definition}{Definition}
\theoremstyle{remark}
\newtheorem{remark}{Remark}

\newtheorem{statment}{Statment}

\newtheorem*{main_theorem}{Main Theorem}

\uccode`ґ=`Ґ \uccode`Ґ=`Ґ \lccode`Ґ=`ґ \lccode`ґ=`ґ %
\uccode`є=`Є \uccode`Є=`Є \lccode`Є=`є \lccode`є=`є %
\uccode`і=`І \uccode`І=`І \lccode`І=`і \lccode`і=`і %
\uccode`ї=`Ї \uccode`Ї=`Ї \lccode`Ї=`ї \lccode`ї=`ї %

\begin{document}
\selectlanguage{english}

\title[] {Structure, minimal generating systems and properties of Sylow 2-subgroups of alternating group}

\author[]{Ruslan Skuratovskii}


\begin{abstract}

   The background of this paper is the following:  search  of the minimal systems of generators for this class of group which still was not founded also problem of representation for this class of group, exploration of systems of generators for Sylow 2-subgroups $Syl_2{A_{2^k}}$ and $Syl_2{A_{n}}$ of alternating group, finding  structure of these subgroups.

The authors of \cite{Dm} didn't proof minimality of finding by them system of generators for such Sylow 2-subgroups of $A_n$ and structure of it were founded only descriptively.
	The aim of this paper is to research the structure of Sylow 2-subgroups and to construct a minimal generating system for such subgroups. In other words, the problem is not simply in the proof of existence of a generating set with elements for Sylow 2-subgroup of alternating group of degree $2^k $ and its constructive proof and proof its minimality. For the construction of minimal generating set we used the representation of elements of group by automorphisms of portraits for binary tree. Also, the goal of this paper is to investigate the structure of 2-sylow subgroup of alternating group more exactly and deep than in \cite{Dm}.
 The main result is the proof of minimality of this generating set of the above described subgroups and also the description of their structure.
Key words:  minimal system of generators; wreath product; Sylow subgroups; group; semidirect product.

\end{abstract}

\maketitle

\section{Introduction }
The aim of this paper is to research the structure of Sylow 2-subgroups of $A_{2^k}$, $A_n$ and to construct a minimal generating system for it. Case of Sylow subgroup where $p=2$ is very special because group $C_2\wr C_2\wr C_2 \ldots \wr C_2 $ admits odd permutations, this case was not fully investigated in \cite{Dm,Sk}. This groups have applications in automaton theory, because if all states of automaton $A$ have output function that can be presented as cycle $(1,2,...,p)$ then group $G_A(X)$ of this automaton is Sylows p-subgroup of the group of all automaton transformations $GA(X)$.
       There was a mistake in a statement about irreducibility that system of $k+1$ elements for $Syl_2(A_{2^k})$ which was in abstract \cite{Iv} on Four ukraine conference
of young scientists in 2015 year.
All undeclared terms are from \cite{Ne, Gr}.
A minimal system of generators for a Sylow 2-subgroups of $A_n$ was found.
 We denote by $v_{j,i}$ the vertex of $X^j$, which has the number $i$.

\section{ Main result  }
 Let $X^{\omega}$ denote a rooted tree, i.e. a connected graph with no cycles
and a designated vertex $v$ called the root.
The $n$-th level, denoted by $X^n$ is defined by the distance $n$,
and $X^0 = {v}$.
The subtree of $X^{*}$ induced by the set of vertices $ \cup^n_{i=0} X^i$
 is called by restriction of rooted tree.
Let's denote by $X^{[k]}$ a labeled regular truncated binary rooted tree (with number of levels from $0$ to $k$ but active states are only from $X^l, 0 \leq l \leq k-1$) labeled by vertex, $v_{j,i} X^{[k-j]} $ -- subtree of $X^{[k]}$ with root in $v_{j,i}$.

Note that the unique vertex $v_{ki}$ corresponds to every word $v$ in alphabet $X$.
For every automorphism $g\in Aut{{X}^{*}}$ and every word $v \in X^{*}$  define the section $g(v) \in AutX^{*}$ of  $g$ at $v$ by the rule:  $g_{(v)}(x) = y$ for $x, y \in X^*$  if and only if $g(vx) = g(v)y$. A restriction $g(v)|X^{[1]} $ is called the state of $g$ at vertex $v$.
Denote by ${{s}_{ki}}(\alpha )$ the state of automorphism at vertex ${{v}_{ki}}$; that is: if ${{\alpha }_{{{v}_{ki}}}}=1$, then ${{s}_{ki}}(\alpha )=1$, otherwise ${{s}_{ki}}(\alpha )=0$.
An automorphism of $X^{[k]}$ with non-trivial state in some of
 $v_{1,1}$, $v_{1,{2}}$, $v_{2,{1}}$,..., $v_{2,4}$, ... ,$v_{m,1}$, ... ,$v_{m,j}$, $ m < k, j \leq 2^m$ is denoted by $\mathop{\beta}_{1,(i_{11},i_{12});...; l,(i_{l1},...,i_{l2^l});...;{m},(i_{m1},...,i_{m2^m})}$ 
where the indexes $l$ are the numbers of levels with non-trivial states, in
parentheses after this numbers we denote a cortege of vertices of this level,  $i_{mj}=0$ if state in $v_{mj}$ is trivial $i_{mj}=1$ in other case. If for some $l$ all $i_{lj}=0$ then cortege $l,(i_{l1} ,..., i_{l 2^l})$ does not figure in indexes of $\beta$. But if numbers of vertices with active states are certain, for example $v_{j,1}$ and $v_{j,s}$ we can use more easy notation $\mathop{\beta}_{j,(1,s);}$. If in parentheses only one index then parentheses can be omitted for instance $\mathop{\beta}_{j,(s);}=\mathop{\beta}_{j,s;}$.
Denote by $\tau_{i,...,j}$ the automorphism, which has a non-trivial vertex permutation (v.p.)
only in vertices $v_{k-1,i}$, ... ,$v_{k-1, j}$, $j  \leq 2^{k-1} $ of the
level $X^{k-1}$.
Denote by $\tau$ the automorphism $\tau_{1,2^{k-1}}$.
Let's consider special elements such that: $ \mathop{\alpha}_{0}=\mathop{\beta}_{0}=\mathop{\beta}_{0,(1,0,...,0)}, \mathop{\alpha}_{1}=\mathop{\beta}_{1}=\mathop{\beta}_{1,(1,0,...,0)}
, \ldots ,\mathop{\alpha}_{l}=\mathop{\beta}_{l}=\mathop{\beta}_{l,(1,0,...,0)} $.
Recall that wreath product of permutations groups is associative construction.

 \begin{lemma} \label{even} 

  Every state from $X^l$, $l<k-1$ determines an even permutation on $X^{k}$.

\end{lemma}
\begin{proof}
Actually every transposition from $X^l$, $l<k-1$ acts at even number of pair of vertexes because of binary tree structure, so it realize even permutation (of set which is vertexes of $X^{k}$) with cyclic structure \cite{Sh}  $(1^{2^{k-1}-2^{k-l-l}},2^{2^{k-l-l}})$ because it formed by the structure of binary tree.
 \end{proof}

\begin{corollary}\label{B_k-1}
Due to Lemma \ref{even} vertices permutations from $Aut X^{[k-1]}=\langle \mathop{\alpha}_{0}. ... , \mathop{\alpha}_{k-2} \rangle$ form a group: ${B_{k-1}}= \underbrace {C_2 \wr ...\wr C_2}_{k-1} $ which acts at $X^{k-1}$ by even permutations. Order of $B_{k-1}$ equal to $2^{2^{k-1}-1}$.
\end{corollary}

Let us denote by $G_k$ the subgroup of $Aut X^{[k]}$ such that $G_k  \simeq  Syl_2  A_{2^k}$
and $W_{k-1}$ the subgroup of $G_k$ that has active states only on $ X^{k-1}$.

\begin{prop} \label{ordW}

An order of ${{W}_{k-1}}$ is equal to ${{2}^{{{2}^{k-1}}-1}},\,\,k > 1$ its structure is $W_{k-1} \simeq (C_2)^{{{2}^{k-1}}-1}$.
\end{prop}

  \begin{proof}
On ${{X}^{k-1}}$ we have ${{2}^{k-1}}$ vertexes where can be group ${{V}_{k-1}}\simeq {{C}_{2}}\times {{C}_{2}}\times ...\times {{C}_{2}}\simeq {{({{C}_{2}})}^{k-1}}$, but as a result of the fact that ${X}^{k-1}$ contains only even number of active states or active vertex permutations (v.p.) in vertexes of $X^{k-1}$, there are only half of all permutations from ${{V}_{k-1}}$ on $X^{k-1}$. So it's subgroup of ${{V}_{k-1}}$: ${{W}_{k-1}}\simeq {}^{C_{2}^{{{2}^{k-1}}}}/{}_{{{C}_{2}}}$. So we can state: $|{{W}_{k-1}}|={2^{k-1}-1}$, $W_{k-1}$ has $k-1$ generators and we can consider ${W}_{k-1}$ as vector space of dimension $k-1$.
 \end{proof}

For example let's consider subgroup $W_{4-1}$ of $A_{2^4}$ its order is $2^{2^{4-1}-1}=2^7$ and $|A_{2^4}|=2^{14}$.

 \begin{lemma}\label{gen} 
 The elements $\tau $ and ${{\alpha }_{0}},...,{{\alpha }_{k-1}}$ generates arbitrary element ${{\tau }_{ij}}$.
\end{lemma}
 \begin{proof} According to [5, 6] 
  the set  ${{\alpha }_{0}},...,{{\alpha }_{k-2}}$ is minimal system of generators of group $Aut{{X}^{[k-1]}}$. Firstly, we shall proof the possibility of generating an arbitrary $\tau_{ij}$, from set $v_{(k-1,i)}$, $1\leq i \leq 2^{k-2} $. 
Since $Aut v_{1,1}X^{[k-2]} \simeq \left\langle {{\alpha }_{1}},...,{{\alpha }_{k-2}} \right\rangle $ acts at ${{X}^{k-1}}$ transitively from it follows existing of an ability to permute vertex with a transposition from automorphism $\tau $ and stands in $v_{k-1,1}$ in arbitrary vertex ${{v}_{k-1,j}},\,\,\,j\le {{2}^{k-2}}$ of $v_{1,1}X^{[k-1]}$, for this goal we act by $\alpha_{k-j}$ at $\tau $: $\alpha_{k-j} \tau \alpha_{k-j} ={{\tau }_{j, 2^{k-2}}}$. Similarly we act at $\tau$ by corespondent $\alpha_{k-i}$ to get $\tau_{i, 2^{k-2}}$ from $\tau $: $\alpha_{k-i}  \tau \alpha_{k-i}^{-1}={{\tau }_{i, 2^{k-2}}}$. Note that automorphisms $\alpha_{k-j}$ and $\alpha_{k-i}, 1<i,j<k-1$ acts only at subtree $v_{1,1}X^{[k-1]}$ that's why they fix v.p. in $v_{k-1, 2^{k-1}}$. Now we see that ${\tau }_{i, 2^{k-2}}{{\tau }_{j, 2^{k-2}}}={{\tau }_{i, j}}$, where $1 \leq i,j < 2^{k-2}$.
To get ${\tau }_{m, l}$ from $v_{1,2}X^{[k-1]}$, i.e. $2^{k-2} < m,l \leq 2^{k-1} $ we use $\alpha_0$ to map ${\tau }_{i, j}$ in ${\tau }_{i+2^{k-2}, j+2^{k-2}}\in v_{1,2} AutX^{[k-1]}$. To construct arbitrary transposition ${\tau }_{i,m}$ from $W_k$ we have to multiply ${\tau }_{1,i} {\tau } {\tau }_{m,2^{k-1}}={\tau }_{i,m}$.
Lets realize natural number of $v_{k,l}$, $1<l<2^k$ in 2-adic system of presentation (binary arithmetic).  Then $l={\delta_{{{l}_{1}}}}{{2}^{m_l}}+{\delta_{{{l}_{2}}}}{{2}^{m_l-1}}+...+{\delta_{{{l}_{m_l+1}}}},\,\, \delta_{l_i} \in \{0,1\}$ where is a correspondence between  ${\delta_{{{l}_{i}}}}$ that from such presentation and expressing of automorphisms: $\tau_{l,2^{k-1}} = \prod_{i=1}^{m_l} \alpha_{k-2-(m_{l}-i)}^{\delta_{{l}_{i}}} \tau  \prod_{i=1}^{m_l} \alpha_{k-2-(m_{l}-i)}^{\delta_{l_i}},  1 \leq m_l \leq k-2$.  In other words $\left\langle {{\alpha }_{0}},...,{{\alpha }_{k-2}},\tau  \right\rangle \simeq {{G}_{k}}$.
 \end{proof}

\begin{corollary}\label{genG_k-1} The elements from conditions of Lemma \ref{gen} is enough to generate basis of $W_{k-1}$.
\end{corollary}

 \begin{lemma}\label{ordG_k} Orders of groups $G_k = \langle   \mathop{\alpha}_{0}, \mathop{\alpha}_{1},
 \mathop{\alpha}_{2},...,\mathop{\alpha}_{k-2}, \mathop{\tau} \rangle $
and $Syl_2(A_{2^{k}})$ are equal to $2^{2^{k}-2}$.
\end{lemma}
 \begin{proof}
In accordance with Legender's formula, the power of 2 in ${{2}^{k}}!$ is $\left[ \frac{{{2}^{k}}}{2} \right]+\left[ \frac{{{2}^{k}}}{{{2}^{2}}} \right]+\left[ \frac{{{2}^{k}}}{{{2}^{3}}} \right]+...+\left[ \frac{{{2}^{k}}}{{{2}^{k}}} \right]=\frac{{{2}^{k}}-1}{2-1}$. We need to subtract 1 from it because we have only $\frac{n!} {2}$ of all permutations as a result: $\frac{{{2}^{k}}-1}{2-1}-1=2^{k}-2$. So $\left| Syl({{A}_{{{2}^{k}}}}) \right|={{2}^{{{2}^{k}}-2}}$.
The same order has group $G_k=B_{k-1} \ltimes W_{k-1}$: $|G_k|=|B_{k-1}|\cdot|W_{k-1}|= |Syl_2 A_{2^k}|$, since order of groups $G_{k}$ according to Proposition \ref{ordW} and the fact that $|B_{k-1}|=2^{2^{k-1}-1}$ is $2^{2^{k}-2}$.
For instance the orders of $Syl_2 (A_8)$, $B_{3-1}$ and $W_{3-1}$: $|W_{3-1}|= 2^{2^{3-1}-1}=2^3=8$, $|B_{3-1}|=|C_2\wr C_2| = 2 \cdot 2^2=2^3$ and according to Legendre's formula, the power of 2 in ${{2}^{k}}!$ is $\frac{{2}^{3}}{2}+ \frac{{2}^{3}}{2^2}+\frac{{2}^{3}}{2^3} -1=6$ so $Syl_2 (A_8) = 2^6=2^{2^k-2}$, where $k=3$. Next example for $A_{16}$: $Syl_2 (A_{16}) =2^{2^4-2}= 2^{14}, k=4$, $|W_{4-1}|= 2^{2^{4-1}-1}=2^7$, $|B_{4-1}|=|C_2\wr C_2\wr C_2| = 2 \cdot 2^2\cdot 2^4 = 2^7$. So we have equality $2^7 2^7 = 2^{14}$ which endorse the condition of this Lemma.
 \end{proof}

\begin{theorem}
\textbf{A maximal 2-subgroups} $G_k$ of $Aut{{X}^{\left[ k \right]}}$ that acts by even permutations on ${{X}^{k}}$ have structure of semidirect product $G_k \simeq  B_{k-1} \ltimes W_{k-1} $ and isomorphic to $Syl_2A_{2^k}$.
\end{theorem}
\begin{proof}
A maximal 2-subgroup of  $Aut{{X}^{\left[ k-1 \right]}}$ is isomorphic to ${{B}_{k-1}}\simeq \underbrace{{{S}_{2}}\wr {{S}_{2}}\wr ...\wr {{S}_{2}}}_{k-1}$ (this group acts on ${{X}^{k-1}}$). A maximal 2-subgroup which has elements with active states only on ${{X}^{k-1}}$ corresponds subgroup ${W}_{k-1}$.
Since subgroups ${B}_{k-1}$ and ${{W}_{k-1}}$ are embedded in $Aut{{X}^{\left[ k \right]}}$, then define an action of ${{B}_{k-1}}$ on elements of ${{W}_{k-1}}$ as ${{\tau }^{\sigma }}=\sigma \tau {{\sigma }^{-1}},\,\,\,\sigma \in {{B}_{k-1}},\,\,\tau \in {{W}_{k-1}}$,
i.e. action by inner automorphism (inner action) from $Aut{{X}^{\left[ k \right]}}$.
Note that ${{W}_{k-1}}$ is subgroup of stabilizer of  ${{X}^{k-1}}$ i.e. ${{W}_{k-1}}<St_{Aut{X}^{[k]}}(k-1)\lhd AutX^{[k]}$ and is normal too $W_{k-1}\lhd AutX^{[k]}$, because conjugation keeps a cyclic structure of permutation so even permutation maps in even. Therefore such conjugation induce automorphism of ${W}_{k-1}$ and $G_k \simeq B_{k-1}\ltimes W_{k-1}$. Since  at ${{X}^{k-1}}$ is ${{2}^{{{2}^{k-1}}}}$ vertexes and half of combinations of active states from  ${X}^{k-1}$ can form even permutation thus $\left| {{W}_{k-1}} \right|={{2}^{{{2}^{k-1}}-1}}$ that is proved in Proposition \ref{ordW}. Using the Corollary \ref{B_k-1} about ${{B}_{k-1}}$ we get order of  ${{G}_{k}}\simeq {{B}_{k-1}} \ltimes {{W}_{k-1}}$ is ${{2}^{{{2}^{k-1}}-1}}\cdot {{2}^{{{2}^{k-1}}-1}}={{2}^{{{2}^{k}}-2}}$.
 Since $G_k$ is maximal 2-subgroup then $G_k \simeq Syl_2A_{2^k}$.

\end{proof}


 \begin{theorem} \label{isomor}  
  The set $S_{\mathop{\alpha}}= \{\mathop{\alpha}_{0}, \mathop{\alpha}_{1},
 \mathop{\alpha}_{2},  ... ,\mathop{\alpha}_{k-2}, \mathop{\tau}\}$
   of elements from subgroup of $AutX^{[k]}$
   generates a group $G_k$ which isomorphic to $Syl_2(A_{2^{k}})$.
\end{theorem}
 \begin{proof}
As we see from Corollary \ref{B_k-1}, Lemma \ref{gen} and Corollary \ref{genG_k-1} group $G_k$ are generated by $S_{\mathop{\alpha}}$ and their orders according to Lemma \ref{ordG_k} is equal. So according to Sylow's theorems 2-subgroup $G_k<A_{2^k}$ is $Syl_2 (A_{2^k})$.
 \end{proof}
 Consequently, we construct a generating system, which contains $k$ elements, that is less than in \cite{Iv}.

The structure of Sylow 2-subgroup of $A_{2^k}$ is the following: $\underset{i=1}{\overset{k-1}{\mathop{\wr }}}\,{{C}_{2}}  \ltimes \prod_{i=1}^{2^{k-1}-1} C_2  $, where we take $C_2$ as group of action on two elements and this action is faithful, it adjusts with construction of normalizer for $Syl_p(S_n)$ from \cite{Weisner}, where it was said that $Syl_2(A_{2^l})$ is self-normalized  in $S_{2^l}$.

Number of such minimal generating systems for ${{G}_{k}}$ is $C_{{{2}^{k-2}}}^{1}C^1_{{2}^{k-2}}=2^{k-2} \cdot 2^{k-2}= 2^{2k-4}$.

Let us present new operation $\boxtimes $ (similar to that is in \cite{Dm}) as a subdirect product of $Syl{{S}_{{{2}^{i}}}}$, $n=\sum_{i=0} ^j a_{k_i} {2}^{i}$, i.e. $Syl{{S}_{{{2}^{{{k}_{1}}}}}}\boxtimes Syl{{S}_{{{2}^{{{k}_{2}}}}}}\boxtimes ...\boxtimes Syl{{S}_{{{2}^{{{k}_{l}}}}}}=Par(Syl{{S}_{{{2}^{{{k}_{1}}}}}}\times Syl{{S}_{{{2}^{{{k}_{2}}}}}}\times ...\times Syl{{S}_{{{2}^{{{k}_{l}}}}}})$, where $Par(G)$ -- set of all even permutations of $G$. Note, that $\boxtimes $ is not associated operation, for instance $ord({{G}_{1}}\boxtimes {{G}_{2}}\boxtimes {{G}_{3}})\,\,\,=\left| {{G}_{1}}\times {{G}_{2}}\times {{G}_{3}} \right|:2$ but $ord({{G}_{1}}\boxtimes {{G}_{2}})\boxtimes {{G}_{3}}\,\,\,=\left| {{G}_{1}}\times {{G}_{2}}\times {{G}_{3}} \right|:4$. For cases $n=4k+1$, $n=4k+3$ it follows from formula of Legendre.

\begin {statment}
 Sylow subgroups of ${{S}_{n}}$ is isomorphic to $Syl_2{{S}_{{{2}^{{{k}_{1}}}}}}\times Syl_2{{S}_{{{2}^{{{k}_{2}}}}}}\times ...\times Syl_2{{S} _{{{2}^{{{k}_{l}}}}}}$, where $n={{2}^{{{k}_{1}}}}+{{2}^{{{k}_{2}}}}+...+{{2}^{{{k}_{l}}}}$, ${{k}_{i}}\ge 0$, $k_{i} < k_{i-1}$. Sylow subgroup $Sy{{l}_{2}}({{A}_{n}})$ has index 2 in $Syl_{2}({{S}_{n}})$ and it's structure: $Syl_2{S_{2^{{{k}_{1}}}}}\boxtimes Syl_2{{S}_{{{2}^{{{k}_{2}}}}}}\boxtimes ...\boxtimes Syl_2{S}_{{2}^{{k}_{l}}}$.
 \end {statment}

\begin{proof} Group $Sy{{l}_{2}}({{S}_{{{2}^{k}}}})$ is isomorphic to $\underbrace{{{C}_{2}}\wr {{C}_{2}}\wr ...\wr {{C}_{2}}}_{k-1}$ \cite{Dm, Sk} and this group is isomorphic to $AutX^{[k]}$, that acts at ${{2}^{k}}$ vertices on ${{X}^{k}}$. In case $Sy{{l}_{2}}({{S}_{n}}),\,\,\, \nexists k,k\in \mathbb{N}:\,\,  n = {{2}^{k}}$ but  $n={{2}^{{{k}_{1}}}}+\,{{2}^{{{k}_{2}}}}+...+{{2}^{{{k}_{l}}}},\,\,{{k}_{i}}\in \mathbb{N}\cup \left\{ 0 \right\}$  is a direct product: $Aut{{X}^{{[{k}_{1}]}}}\times ...\times Aut{{X}^{{[{k}_{l}]}}}$ \cite{Sh}. Power of 2 in this $|A_n|$: $\frac{2^{k_1}}{2}+\frac{2^{k_1}}{4}+...+1+\frac{2^{k_1}}{2}+\frac{2^{k_2}}{4}+...+1+\frac{2^{k_l}}{2}+...+1 - (\underbrace{1 + ... +1}_l)$. But only half of these automorphism determines even permutations on $X^k$ so we have been subtracted 1 $l$ times. For counting order of $Aut{{T}_{{{k}_{1}}}}\times ...\times Aut{{T}_{{{k}_{l}}}}$ we have to take in consideration that $AutX^{[k_i]}$ has active states on levels $X^0$, $X^1$, ... ,$X^{k_i-1}$ and
hasn't active states on $X^{k_i}$ so $|AutX^{[k_i]}|=1\cdot 2 \cdot 2^2 \cdot...\cdot 2^{k_i-1}=2^{2^{k_i}}-1$.

 Such presentations is unique because it determines by binary number presentation  $n={{2}^{{{k}_{1}}}}+\,\,...\,\,+{{2}^{{{k}_{l}}}},\,\,{{k}_{i}}\in \mathbb{N}\cup \left\{ 0 \right\}$, exactly for this presentation corresponds decomposition in direct product $Sy{{l}_{2}}{{S}_{{{2}^{{{k}_{1}}}}}}\times Sy{{l}_{2}}{{S}_{{{2}^{{{k}_{2}}}}}}\times ...\times Sy{{l}_{2}}{{S}_{{{2}^{{{k}_{l}}}}}}$ and this decomposition determines Sylow 2-subgroup of maximal order by unique way. It is so, since for alternating decomposition ${{2}^{{{k}_{1}}}}={{2}^{l}}+{{2}^{l}}$ order: $\left| Sy{{l}_{2}}{{S}_{{{2}^{{{k}_{1}}}}}} \right|>\left| Syl_{2}{{S}_{{{2}^{l}}}}\times Sy{{l}_{2}}{{S}_{{{2}^{l}}}} \right|$ or more precisely $\left| Sy{{l}_{2}}{{S}_{{{2}^{{{k}_{1}}}}}} \right|=2\left| Sy{{l}_{2}}{{S}_{{{2}^{l}}}}\times Sy{{l}_{2}}{{S}_{{{2}^{l}}}} \right|$ it following from structure of group binary trees.

 For instance $Syl_{2}{{S}_{22}}\simeq Sy{{l}_{2}}{{S}_{16}}\times Sy{{l}_{2}}{{S}_{4}}\times Sy{{l}_{2}}{{S}_{2}}$ and correspondent orders ${{2}^{15}},\,\,{{2}^{3}},\,\,2$ so it's order is ${{2}^{19}}={{2}^{15}}\cdot {{2}^{3}}\cdot 2$, on the other hand order of  $Sy{{l}_{2}}{{S}_{22}}$  by formula of Legendre is ${{2}^{19}}={{2}^{11+5+2+1}}$, analogously $Sy{{l}_{2}}{{S}_{24}}\simeq Sy{{l}_{2}}{{S}_{16}}\times Sy{{l}_{2}}{{S}_{8}}$ and ${{2}^{22}}={{2}^{15}}\cdot {{2}^{7}}$, on the other hand ${{2}^{22}}={{2}^{12+6+3+1}}$. Let us prove that such decompositions of $Sy{{l}_{2}}{{S}_{n}}$ and $Aut{{T}_{n}}$ are unique in accord with $n={{2}^{{{k}_{1}}}}+{{2}^{{{k}_{2}}}}+...+{{2}^{{{k}_{m}}}}$, where ${{k}_{1}}>{{k}_{2}}>...>{{k}_{m}}\ge 0$.

Really in accord with Legender's formula $n!$ is divisible by 2 in power $n-m={{2}^{{{k}_{1}}}}+{{2}^{{{k}_{2}}}}+...+{{2}^{{{k}_{m}}}}-m$ so $\left| Syl_2 {{S}_{22}} \right|={{2}^{n-m}}$, group $Aut{{T}_{{{k}_{1}}}}\times ...\times Aut{{T}_{{{k}_{m}}}}$, where ${{k}_{1}}>{{k}_{2}}>...>{{k}_{m}}\ge 0$, also has order ${{2}^{{{2}^{{{k}_{1}}-1}}+...+{{2}^{{{k}_{m}}-1}}}}={{2}^{n-m}}$ it follows from formula of geometric progression. Such decomposition $Sy{{l}_{2}}{{S}_{{{2}^{{{k}_{1}}}}}}\times Sy{{l}_{2}}{{S}_{{{2}^{{{k}_{2}}}}}}\times ...\times Sy{{l}_{2}}{{S}_{{{2}^{{{k}_{l}}}}}}$ determines 2-Sylows subgroup of maximal order. It is so, since for alternating decomposition ${{2}^{{{k}_{1}}}}={{2}^{l}}+{{2}^{l}}$ order: $\left| Sy{{l}_{2}}{{S}_{{{2}^{{{k}_{1}}}}}} \right|>\left| Sy{{l}_{2}}{{S}_{{{2}^{l}}}}\times Sy{{l}_{2}}{{S}_{{{2}^{l}}}} \right|$ or more precisely $\left| Sy{{l}_{2}}{{S}_{{{2}^{{{k}_{1}}}}}} \right|=2\left| Sy{{l}_{2}}{{S}_{{{2}^{l}}}}\times Sy{{l}_{2}}{{S}_{{{2}^{l}}}} \right|$ it
following from structure of group binary trees ($Aut{{T}_{l}}\times Aut{{T}_{l}}$) which correspondent for $Sy{{l}_{2}}{{S}_{{{2}^{l}}}}\times Sy{{l}_{2}}{{S}_{{{2}^{l}}}}$ and formula of geometric progression. It also follows from nesting (embeding) of Sylows subgroups. If $n={{2}^{{{k}_{1}}}}+{{2}^{{{k}_{2}}}}={{2}^{l}}+{{2}^{l}}+{{2}^{{{k}_{3}}}},$ where ${{2}^{{{k}_{1}}}}={{2}^{l}}+{{2}^{l}}$, then $Aut{{T}_{n}}\simeq Aut{{T}_{{{k}_{1}}}}\times Aut{{T}_{{{k}_{2}}}}$, if not then $Sy{{l}_{2}}{{A}_{n}}$ don't contains $Sy{{l}_{2}}{{A}_{{{2}^{{{k}_{1}}}}}}$.
\end{proof}

\begin{remark}
If $n=2k+1$ then
$Syl_2(A_n) \cong Syl_2(A_{n-1})$ and
$Syl_2(S_n) \cong Syl_2(S_{n-1})$.
\end{remark}
 \begin {proof}
 Orders of these subgroups are equal since according to Legender's formula which counts power of 2 in $(2k+1)!$ and $(2k)!$ we get that these powers are equal. So these maximal 2-subgroups are isomorphic. From Statement 1 can be obtained that vertex with number $2k+1$ will be fixed to hold even number of transpositions on $X^{k_1}$ from decomposition of $n$ which is in Statement 1. For instance $Syl_2(A_{7})\simeq Syl_2(A_{6})$ and by the way $Syl_2(A_{6})\simeq C_2 \wr C_2 \simeq D_4$, $Syl_2(A_{11})\simeq Syl_2(A_{10}) \simeq C_2 \wr C_2 \wr C_2 $.
\end {proof}

\begin {definition}
 We call index of automorphism $\beta$ on $X^l$ number on active states of $\beta$ on $X^l$.
\end {definition}

\begin {definition} Define a generator of type \texttt{T} as an automorphism ${{\tau }_{{{i}_{0}},...,{{i}_{{{2}^{k-1}}}};{{j}_{{{2}^{k-1}}}},...,{{j}_{{2}^k}}}}$, that has even index at ${{X}^{k-1}}$ and ${{\tau }_{{{i}_{0}},...,{{i}_{{{2}^{k-1}}}};{{j}_{{{2}^{k-1}}}},...,{{j}_{{{2}^{k}}}}}}\in S{{t}_{Aut{{X}^{k}}}}(k-1)$,
and it consists of odd number of active states in vertexes ${{v}_{k-1,j}}$ with number $ j \leq  2^{k-2}$ and odd number in vertices ${{v}_{k-1,j}}$, ${2^{k-2}} < j \leq {2^{k-1}}$. Set of such elements denote \texttt{T}.
\end {definition}
\begin {definition}
A combined generator is such an automorphism ${{\beta }_{{l}; \tilde{\tau } }}$, that the restriction ${{\beta }_{{l};\tilde{\tau } }}\left| _{{{X}^{k-1}}} \right.$ coincides with ${{\alpha }_{{l}}}$ and $Rist_{<\beta_{{i_l}; \tilde{ \tau}  }>}(k-1)=\left\langle \tau' \right\rangle $, where $\tau' \in $\texttt{T}. The set of such elements is denoted by \texttt{CG}.
\end {definition}

\begin {definition} A combined element is such an automorphism $ {{\beta }_{1,{{i}_{1}};2,{{i}_{2}};...;k-1,{{i}_{k-1}}; \tilde{\tau }}}$, that it's restriction
${{\beta}_{1,{{i}_{1}};2,{{i}_{2}};...;k-1,{{i}_{k-1}};\tilde{\tau }}}\left|_{{{X}_{k-1}}} \right.$ coincide with one of elements that can be generated by ${{S}_{\alpha}}$ and $Rist_{<{{\beta}_{1,{{i}_{1}};2,{{i}_{2}};...;k-1,{{i}_{k-1}}; \tilde{ \tau} }}>}(k-1)  =\left\langle \tau' \right\rangle $ where $\tau' \in $\texttt{T}. Set of such elements is denoted by \texttt{C}.
\end {definition}
In other word elements $g \in$\texttt{C} and $g' \in$\texttt{CG} on level  ${{X}^{k-1}}$ have such structure as element and generator of type \texttt{T}.

The minimal number of elements in a generating set  $S$
of G we denote by rk$G$ and call the rank of $S$.

By distance between vertexes we shall understand usual distance at graph between its vertexes.
By distance of vertex permutations we shall understand maximal distance between two vertexes with active states from $X^{k-1}$.

\begin{lemma} \label{Lemma about keeping of distance} A vertices permutations on ${X}^{k}$ that has distance ${{d}_{0}}$ can not be generated by vertex permutations with distance ${{d}_{1}}:\,\,{{d}_{1}}<{{d}_{0}}$.
\end{lemma}
\begin{proof}
Really vertex permutation ${{\tau }_{ij}}:\,\,\,\,\,\rho ({{\tau }_{ij}})={{d}_{0}},\,\,\,{{d}_{0}}<{{d}_{1}}$ can be mapped by automorphic mapping only in permutation with  distance ${{d}_{0}}$ because automorphism keep incidence relation and so it possess property of isometry. Also multiplication of portrait of automorphism ${{\tau }_{ij}}:\,\,\,\,\,\rho ({{\tau }_{ij}})={{d}_{1}}$ give us automorphism with distance ${{d}_{1}}$, it follows from properties of group operation and property of automorphism to keep a distance between vertices on graph.
\end{proof}

\begin{lemma} \label{about transposition} An arbitrary automorphism $\tau{'} \in$\texttt{T} (or in particular $\tau $ ) can be generated only with using odd number of automorphisms from \texttt{C} or \texttt{T}.
\end{lemma}
\begin{proof}
Let us assume that there is no such element which has distance of v. p. $2k$  ${{\tau }_{ij}}$ then accord to Lemma \ref{Lemma about keeping of distance} it is imposable to generate are pair of transpositions $\tau{'} $ with distance $\rho ({{\tau }_{ij}})=2k$ since such transpositions can be generated only by ${{\tau }_{ij}}$ that described in the conditions of this Lemma: $i\le {{2}^{k-2}},\,\,\,j>{{2}^{k-2}}$. Combined element can be decomposed in product $ \tau \dot {\beta }_{{i}_{l}} = {{\beta }_{{{i}_{l}};\tau }} $ so we can express by using $\tau$ or using a product where odd number an elements from \texttt{T} or \texttt{C}. If we consider product $P$ of even number elements from \texttt{T} then automorphism $P$ has even number of active states in vertexes ${{v}_{k-1,i}}$ with number $ i \leq  2^{k-2}$ so $P$ does not satisfy the definition of generator of type \texttt{T}.
\end{proof}
\begin{corollary} \label{About generating distance} Any element of type \texttt{T} can not be generated by ${{\tau }_{ij}}\in Aut {{v}_{1,1}}{{X}^{[k-1]}}$ and ${{\tau }_{ml}}\in Aut {{v}_{1,2}}{{X}^{[k-1]}}$. The same corollary is true for a combined element.
\end{corollary}
\begin{proof}
It can be obtained from the Lemma \ref{Lemma about keeping of distance} because such ${\tau_{ij}}\in Aut {{v}_{1,1}}{{X}^{[k-1]}}$ has distance less then $2k-2$ so it does not satisfy conditions of the Lemma \ref{Lemma about keeping of distance}. I.e. $\tau{'}$ can not be generated by vertices permutations with distance between vertices less than $2k-2$ such distance has only automorphisms of type \texttt{T} and \texttt{C}. But elements from $ Aut {{v}_{1,1}}{{X}^{[k-1]}}$ do not belongs to type \texttt{T} or \texttt{C}.  
\end{proof}

It's known that minimal generating system ${{S}_{\alpha}}$ of $Aut{{T}_{k-1}}$ has $k-1$ elements \cite{Gr}, so if we complete it by $\tau$ (or element of type \texttt{T} or change one of its generators on $\beta \in \text{T}$ we get system ${{S}_{\beta }}$: ${{G}_{k}}\simeq \left\langle {{S}_{\beta }} \right\rangle $ and $ |{{S}_{\beta }}|= k$.
So to construct combined element we multiply generator ${{\beta }_{i}}$ of ${{S}_{\beta }}$ (or arbitrary element $\beta $ that can be express from ${{S}_{\beta }}$) by the element of type $\tau $, i.e. we take $\tau' \cdot {{\beta }_{i}}$ instead of ${{\beta }_{i}}$ and denote it  ${{\beta }_{i;\tilde{\tau} }}$. It's equivalent that $Ris{{t}_{{{\beta }_{1,i;2,...,j;\tau }}}}(k)=\left\langle \tau' \right\rangle $, where $\tau'$ -- generator of type \texttt{T}. 

\begin{lemma} \label{About not closed set of element of type T} Sets of elements of types \texttt{T}, \texttt{C}  are not closed by multiplication and raising to even power.
\end{lemma}

\begin{proof}
Let  $\varrho, \rho \in$ \texttt{T} (or \texttt{C}) and $\varrho, \rho = \eta$. The numbers of active states in vertices $v_{k-1, i}$, $1 \leq i \leq 2^{k-2}$ of $\varrho$ and $\rho$ sums by $mod 2$, numbers of active states in vertices on $v_{k-1, i}$, $1 \leq i \leq 2^{k-2}$ of $\varrho$ and $\rho$ sums by $mod 2$ too. Thus $\eta$ has even numbers of active states on these corteges.
 So $ RiS{{t}_{\left\langle \eta  \right\rangle }}(k-1)$ doesn't contain elements of type $\tau $ so $\eta \notin $\texttt{T}. Really if we raise the element ${{\beta }_{1,{{i}_{1}};2,{{i}_{2}};...;k-1,{{i}_{k-1}};\tau }}\in $\texttt{T} to even power or we evaluate a product of even number of multipliers from \texttt{C} corteges ${{\mu }_{0}}$ and ${{\mu }_{1}}$ permutes with whole subtrees ${{v}_{1,1}}{{X}^{[k-1]}}$ and ${{v}_{1,2}}{{X}^{[k-1]}}$, then we get an element $g$ with even indices of ${{X}^{k}}$ in $v_{1,1}{{X^{k-1}}}$ and $v_{1,2}{{X^{k-1}}}$. Thus $g \notin $\texttt{T}. Consequently elements of \texttt{C} do not form a group, and the set \texttt{T} as a subset of \texttt{C} is not closed too.
\end{proof}
We have to take into account that all elements from \texttt{T} have the same main property to consists of odd number of active states in vertices ${{v}_{k-1,j}}$ with index $j\le {{2}^{k-2}}$ and odd number in vertices with index $j:$ ${{2}^{k-2}}<j\le {{2}^{k-1}}$.

Let ${S_{\alpha }}=\left\langle {{\alpha }_{0}},\,{{\alpha }_{1}},...,{{\alpha }_{k-2}} \right\rangle $, then $S_{\alpha }^{'}=\left\langle {{\alpha }_{0}},{{\alpha }_{0;1,({i}_{11},{i}_{12});}},...,{{\alpha }_{{0};1,({{i}_{11},...});...,k-2,({i}_{k-2,1},...) }} \right\rangle $, $\left\langle {{S}_{\alpha }} \right\rangle =\left\langle S_{\alpha }^{'} \right\rangle =Aut{{X}^{[k-1]}}$, $rk\left( {{S}_{a}} \right)=k-1$ where $rk\left( S \right)$ is the rank of group which is equal to number generators of its least generating system $S$ \cite{Bog}.
Let $S_{\beta }={{S}_{\alpha }}\cup \tau_{i...j} $, $\tau_{i...j} \in$\texttt{T},
${{S}_{\beta }^{'}}=\left\langle
{{\beta }_{0}},{\beta }_{1,({1});\tau^x },..., {{\beta }_{k-2, (1);\tau^x }}, \tau \right\rangle $, $x \in \{0,1\}$, note if $x=0$ then ${\beta }_{{l,({1})};\tau^x} = {\beta }_l$.
   In $S_{\beta }^{'}$ can be used $\tau  $ instead of $\tau_{i...j}\in T $ because it is not in principle for proof so let $S_{\beta }=S_{\alpha }\cup \tau $ hence $\left\langle {{S}_{\beta }} \right\rangle ={{G}_{k}}$.

Let us assume that $ S_{\beta }^{'}$ does not contain $\tau$ and so had rank $k-1$.
 To express element of type \texttt{T} from ${{S}_{\beta }^{'}}$ we can use a word ${{\beta }_{i,\tau }}\beta _{i}^{-1}=\tau $ but if ${{\beta }_{i,\tau }}\in {{S}_{\beta }^{'}}$ then ${{\beta }_{{{i}}}}\notin S_{\beta }^{'}$ in contrary case rank of $S_{\beta }^{'}$ is $k$.
 So we can not express word ${{\beta }_{i,\tau }}\beta _{i}^{-1}\left| _{{X}^{[k-1]}} \right.=e$ then ${{\beta }_{i,\tau }}\beta _{i}^{-1}=\tau $
  so we have to find relation in restriction of group ${{G}_{k}}$ on ${{X}^{[k-1]}}$. Next lemma investigate existing  of such word.  We have to take in consideration that $ {{G}_{k}}\left| _{{{X}^{k-1}}} \right.={{B}_{k-1}}\simeq Aut{{X}^{[k-1]}} $.

\begin{lemma} \label{Lem About product elem of C}
An element of type \texttt{T} cannot be expressed by $ S_{\beta }^{'}\backslash \{ \tau \} $ with using a product, where are even number of the elements have combine type.
\end{lemma}

\begin{proof}
In this case an element of type \texttt{T} can be expressed only without using elements from \texttt{C}. Note that element of type $\tau $ belongs to $Ri\text{S}{{\text{t}}_{Aut{{T}_{k}}}}(k-1)$. Let us show that automorphism $\alpha :\,\,\,\alpha \left| _{{{X}^{k-1}}} \right.\in ~Ri\text{S}{{\text{t}}_{G}}(k)$, i.e. that is trivial at restriction on ${X}^{[k-1]}$ and belongs to
\texttt{T} can not be expressed from $S_{\beta }^{'}$ without using even number of elements from \texttt{C} because word ${{\alpha }_{{{i}_{1}}}}{{\alpha }_{{{i}_{2}}}}...{{\alpha }_{{{i}_{n}}}}=e$ of $Aut{{T}_{k-1}}$ is word that may be reduced to relation in $\underbrace{{{S}_{2}}\wr ...\wr {{S}_{2}}}_{k-1}$. It known that in wreath product $\wr _{j=1}^{k}{{\mathcal{C}}_{2}}$ holds a relations $\beta _{i}^{2^m}=e$ and $[\beta _{i}^{-1}{{\beta }_{j}}{{\beta }_{i}},\sigma _{j}^{{}}]=e,\,\,i<j$, $\left\langle {{\beta }_{j}},{{\sigma }_{j}} \right\rangle \simeq {{G}_{j}}$, \cite{Sk, DrSku} or more specially $\left[ \beta _{m}^{i}{{\beta }_{{{i}_{n}}, \tau }}\beta_{m}^{-i},\,\beta _{m}^{j}{{\beta }_{{{i}_{k}}, \tau }}\beta _{m}^{-j} \right]=e,\,\,\,i\ne j$, where $n,k,m$ are number of group in $\wr _{j=1}^{k}{\mathcal{C}}_{2}$ ($m<n$, $m<k$) \cite{Sk}, ${\beta }_{{{i}_{n}}}$, ${{\beta }_{m}},\,\,{{\beta }_{{{i}_{k}}}}$ are generators from ${{S}_{^{\beta }}}$.
As we see such relations $r_i$ has structure of commutator so logarithm \cite{K} of every $r_i$ by every generator $\beta_i \in S_\beta$ is zero because every exponent entering with different signs so every $r_i$ has logarithm 0 by every element from \texttt{C} hence is not equal to element of T because Lemma \ref{About not closed set of element of type T}.
 Really product of even number of
  arbitrary elements of \texttt{T} according to Lemma \ref{About not closed set of element of type T} (analogously  \texttt{C}) isn't element of \texttt{T}.
So other way to express automorphism $\alpha :\,\,\,\alpha \left| _{{{X}^{k-1}}} \right.\in ~Ri\text{S}{{\text{t}}_{{{B}_{k-1}}}}(k-1)$ doesn't exist according to relations in ${{G}_{k}}$. Cyclic relation in this group has form $\beta _{{{i}_{0}};\tau }^{{{2^x}}}=e$, analogously for $\beta _{1,{{i}_{1}};2,{{i}_{2}};...;k-1,{{i}_{k-1}};\tau }^{2^y}=e$ because it is 2-group, so we need to raise generator to even power.
So another way to express automorphism $\alpha :\,\,\,\alpha \left| _{{{X}^{k-1}}} \right.\in ~Ri\text{S}{{\text{t}}_{{{B}_{k-1}}}}(k-1)$ does not exist according to relations in ${{G}_{k}}$.
\end{proof}

Let us assume existence of generating system of rank $k-1$ for $Syl_2(A_{2^k})$ that in general case has form $S_{\beta }^{*}(k-1)=\left\langle
 \mathop {\beta}_0, \mathop {\beta}_{0;1,(i_{11},i_{12}); \pi_1 } , ... ,  \beta_{0;...;i,(i_{i1},...,i_{ij});\pi_i } , ... ,  \beta_{0;...;k,(i_{k-11},...,i_{k-1j}) };\pi_k 
   \right\rangle, \ 0<i<k-1, j \leq 2^{i}, \pi_1 \in T,  $, where $\pi_i$ is the cortege of vertices from $X^{k-1}$ with non trivial states which realize permutation with distance $2(k-1)$. In other word if element $\pi_i \in$ \texttt{T} then $\beta_{0;...;i,(i_{i1},...,i_{ij}); \pi_i } = \pi_i \alpha_{0;...;i,(i_{i1},...,i_{ij}) }$.
    From it follows $\mathop {\beta}_{0;1,(i_{11},i_{12}); ... ;{m},(i_{m1},...,i_{mj};\pi_m) }\mid_{X^{[k-1]}} = \mathop {\alpha}_{0;1,(i_{11},i_{12}); ... ,{m},(i_{m1},...,i_{mj}) } \in S^{'}_{\alpha}$.

Note, that automorphisms from system $S^{*}_{\beta}(k-1)$ generate on truncated rooted tree \cite{Ne} $X^{[k-1]}$ group $\langle S^{*}_{\beta}(k-1) \rangle \mid_{X^{[k-1]}} \simeq \wr^{k-1}_{i=1} C_2^{(i)}\simeq B_{k}$.
\begin{theorem} \label{Th about general relation}
Any element of type \texttt{T} can not be expressed by elements of $S^{*}_{\beta}(k-1)$.
\end{theorem}
\begin{proof}
It is necessary to express an automorphism $ \theta$ of type \texttt{T} express such automorphism which has zero indexes of $X^0, ... , X^{k-2}$, this conclusion follows from structure of elements from \texttt{T}. It means that word $w$ from letters of $S^{*}_{\beta}(k-1)$ such that $w = \theta$ is trivial in group $B_k$, that arise on restriction of $\langle S^{*}_{\beta}(k-1)\rangle$ on $X^{[k-1]}$, as well restriction $ G_k \mid_{X^{[k-1]}}\simeq B_k$.

 Every relation from $B_{k}$ can be expressed as a product of words from the normal closure $R^{B_{k}}$ of the set of constitutive relations of the group $B_{k}$ \cite{Bog}.
    But defined relations $r_i$ of $B_{k}$ have form of commutators \cite{DrSku, Sk}
  so the number of inclusions of every multiplier is even and as follows from lemma \ref{About not closed set of element of type T}
that $r_i \mid_{[X^k]} \notin$ \texttt{T}.
             Really in wreath product $\wr _{j=1}^{k}{{\mathcal{C}}^{(j)}_{2}}\simeq B_{k-1}$ holds a constitutive relations
      $\alpha_{i}^{2^m}=e$ and
  $\left[ \alpha _{m}^{i}{{\alpha }_{{{i}_{n}} }}\alpha_{m}^{-i},\,\alpha_{m}^{j}{\alpha }_{{{i}_{k}} }\alpha_{m}^{-j} \right]=e,\,\,\,i\ne j$, where $n,k,m$ are number of groups in $\wr _{j=1}^{k}{{\mathcal{C}}^{(j)}_{l}}$ ($m<n$, $m<k$) \cite{Sk, DrSku}, where
 ${\alpha }_{{{i}_{n}}}$, ${\alpha }_{{{i}_{n}}}$, ${{\alpha }_{m}},\,\,{{\alpha }_{{{i}_{k}}}}$ are generators of the $B_{k-1}$ from ${{S}^{'}_{^{\alpha }}}$.
    So it give us a word $\left[ \beta _{m}^{i}{{\beta }_{{{i}_{n}},{\pi} }}\beta_{m}^{-i},\,\beta _{m}^{j}{\beta }_{{{i}_{k}}, {\pi} }\beta_{m}^{-j} \right],\,\,\,i\ne j$ that could be an automorphism $ \theta$ but does not belongs to \texttt{T} because the word has structure of commutator or belongs to normal closure $R^{B_{k-1}}$ so has logarithm 0 by every element, where
   ${\beta }_{{{i}_{n}}}$, ${\beta_m},\,\,{{\beta }_{{{i}_{k}}}},\, {{\beta }_{{{i}_{n}}, \pi }}$ are generators of $G_k$ from
    $S^{*}_{\beta}(k-1)$.

    Let ${{\beta}_{1,1,(i_{11},i_{12}), \ldots ,i,(i_{i1},...,i_{i2^i});\pi_i}}$ is arbitrary element of type \texttt{C}, where $\pi_i$ -- cortege of vertexes from $X^{k-1}$ with non trivial states which realize permutation with distance $2(k-1)$.

 The case where $\theta = \beta_{{k-1}; {\pi_i}} \in$ \texttt{T} can be expressed by multiplying arbitrary ${{\beta}_{1,(i_{11,12}), \ldots ,i,(i_{i1},...,i_{i2^i});\pi_i}}$ on $\beta_{1,(i_{11,12}), \ldots ,i,(i_{i1},...,i_{i2^i})}^{-1}$ means that such system has rank more then $k-1$ what is contradiction.
   Really if ${{\beta}_{1,(i_{11,12}), \ldots ,i,(i_{i1},...,i_{i2^i})}}$ and ${{\beta}_{1,(i_{11,12}), \ldots ,i,(i_{i1},...,i_{i2^i});\pi_i}} \,  \in S_{\beta}^{*}$ then ${\beta}_{1,(i_{11,12}), \ldots ,i,(i_{i1},...,i_{i2^i})}  |_{X^{[k-1]}} = {{\beta}_{1,(i_{11,12}), \ldots ,i,(i_{i1},...,i_{i2^i});\pi_i}}|_{X^{[k-1]}} $ i.e. these elements are mutually inverse at restriction on $X^{[k-1]}$,
    but it means that in restriction $S^{*}_{\beta}(k-1)$ on $X^{[k-1]}$ that corresponds to the generating system $S^{'}_{\alpha}(k-1)$ for $AutX^{[k-1]} \simeq B_{k-1} $ we use two equal generators. So it has at least $k$ generators,
    because rank of $Aut X^{[k-1]}$ is equal to $k-1$ according to lemma \ref{B_k-1}.

The subcase of this case where ${{\beta}^{-1}_{1,(i_{11},i_{12}), \ldots ,l,(i_{l1},...,i_{l2^i})}}$ 
 can be expressed from $S^{*}_{\beta}(k-1)$ as a product of its generators has the same conclusion. Really if we can generate arbitrary element from $B_{k}$ by generators from $S^{*}_{\beta}(k-1)$ then $k-1$ generators is contained in system $S^{*}_{\beta}(k-1)$ but we have a further ${{\beta}_{1,(i_{11},i_{12}), \ldots ,i,(i_{i1},...,i_{i2^i}); {\pi}}}$. In other words if arbitrary element ${{\beta}_{1,(i_{11},i_{12}), \ldots ,l,(i_{l1},...,i_{l2^i})}}$ of the $B_{k-1}$ does not contains in $S^{*}_{\beta}(k-1)$ but can be expressed from it then $S^{*}_{\beta}(k-1)$ has at least $k-1$ elements exclusive of ${{\beta}_{1,(i_{11},i_{12}), \ldots ,i,(i_{i1},...,i_{i2^i}); {\pi}}}$  then its rank at least $k$.
\end{proof}
\begin{corollary}\label{generating pair} A necessary and sufficient condition of expressing an element  $\tau $ from $S_{\beta }^{'}$  is existing of pair:  ${{\beta }_{{{i}_{m}};\tau }},\,\,\,\,{{\beta }_{{{i}_{_{m}}}}}$ in $\langle S_{\beta}^{'} \rangle$. So rank of a generating system of $G_k$ which contains $S_{\beta }^{'}$ is at least $k$.
\end{corollary}
\begin{proof} Proof can be obtained from Lemma \ref{Lem About product elem of C} and
Lemma \ref{Th about general relation} from which we have that element of type \texttt{T} cannot be expressed from $\left[ \beta _{m}^{i}{{\beta }_{{{i}_{n}}.\tau }}\beta _{m}^{-i},\,\beta _{m}^{j}{{\beta }_{{{i}_{k}}.\tau }}\beta _{m}^{-j} \right]=e,\,\,\,i\ne j$ because such word has even number of elements from \texttt{C}. Sufficient condition follows from  formula ${{\beta }_{{{i}_{m}};\tau }}\beta _{{{i}_{_{m}}}}^{-1}=\tau $.
\end{proof}
\begin{lemma} \label{rk} A generating system of  ${{G}_{k}}$ contains $S_{\alpha}^{'}$ and has at least $k-1$ generators.
\end{lemma}
\begin{proof} The subgroup ${{B}_{k-1}}<{{G}_{k}}$ is isomorphic to $Aut{{X}^{k-1}}$ that has a  minimal system of generators  of $k-1$ elements \cite {Gr}. Moreover, the subgroup ${{B}_{k-1}} \simeq {}^{{{G}_{k}}}/{}_{{{W}_{k-1}}}$, because $G_{k} \simeq B_{k-1}\ltimes {W}_{k-1}  $, where ${W}_{k-1} \vartriangleright G_k$. As it is well known that if $H\lhd G$ then $\text{rk}(G)\ge \text{rk}(H)$, because all generators of $G_{k}$ may belongs to different quotient classes \cite{Magn}.
\end{proof}
As a corollary of last lemma and lemma \ref{Lem About product elem of C} we see that generating system of rank $k-1$ does not exist.
\begin{statment} \label{comm}
Frattini subgroup $ \phi(G_k)= {{G_k}^{2}}\cdot [G_k,G_k]= {{G_k}^{2}} $ acts by all even permutations on ${{X}^{l}},\,\,\,0\le l<k-1$
and by all even permutations on $ X^{k}$ except for those  from \texttt{T}.
\end{statment}
\begin{proof}
Index of the automorphism $\alpha^2 $, $\alpha  \in {{S}_{\beta }}$ on $X^l$ is always even. Really the parity of the number of vertex permutations at $X^l$ in the product $({\alpha }_i {\alpha }_j)^2$, ${\alpha }_i, {\alpha }_j \in  S_{\alpha }$, $i,j<k$ is determined exceptionally by the parity of the numbers of active states at ${{X}^{l}}$ in $\alpha $ and $\beta $ (independently of the action of v.p. from the higher levels). On $X^{k-1}$ group $G^2$ contains all automorphisms of form $\tau_{1 i},  i\leq 2^{k-1}$ which can be generated in such way $({\alpha }_{k-2} \tau_{12})^2= \tau_{1234}$, $\tau_{12}\tau_{1234} = \tau_{34}$,  $({\alpha }_{k-i} \tau_{12})^2= \tau_{1, 2, 1+2^{k-i}, 2+2^{k-i}}$ then $\tau_{1, 2, 1+2^{k-i}, 2+2^{k-i}} \tau_{12} =\tau_{1+2^{k-i}, 2+2^{k-i}}$. In such way we get set of form ${\tau_{12}, \tau_{23},, \tau_{34}, ... ,\tau_{2^{k-1}-1,2^{k-1}}}$. This set is the base for $W_{k-1}$.

The parity of the number of vertex permutations at $X^l$ in the product ${\alpha }_i$ or ${\alpha }_i {\alpha }_j$, ${\alpha }_i,{\alpha }_j \in  S_{\alpha }$) is determined exceptionally by the parity of the numbers of active states at ${{X}^{l}}$ in $\alpha $ and $\beta $ (independently of the action of v.p. from the higher levels). Thus $[\alpha ,\beta ]=\alpha \beta {\alpha }^{-1}{\beta }^{-1}$ has an even number of v. p. at each level. Therefore, the commutators of the generators from ${{S}_{\alpha }}$ and elements from $G^2$ generate only the permutations with even number of v. p. at each ${{X}^{l}}$, ($0\le l\le k-2$).

Let us consider ${{\left( {{\alpha }_{0}}{{\alpha }_{l}} \right)}^{2}}={{\beta }_{l({{1,2}^{l-1}}+1)}}$. Conjugation by the element ${{\beta }_{1(1,2)}}$ (or ${{\beta }_{i(1,2)}},\,\,0<i<l$)  give us ability to express arbitrary coordinate $x:\,\,1\le x \leq {{2}^{l-1}}$ where $x=2^{k-1}-i$,
 i.e. from the element ${{\beta }_{l({{1,2}^{l-1}}+1)}}$ we can express ${{\beta }_{l(x{{,2}^{l-1}}+1)}}$. For instance $x={{2}^{j-1}}+1$, $j<l$: ${{\beta }_{l-j(1,2)}}{{\beta }_{l({{1,2}^{l-1}}+1)}}{{\beta }_{l-j(1,2)}}={{\beta }_{l({{2}^{j-1}+1},{{2}^{l-1}}+1)}}$. If $x={{2}^{l-j}}+2$ than to  realize every shift on $x$ on set $X^l$ 
the element
${{\beta }_{l({{1,2}^{l-1}}+1)}}$ should to be conjugated by such elements ${{\beta }_{l-j(1,2)}} {{\beta }_{l-1(1,2)}}$. So in such way can be realized every ${{\beta }_{l(x{{,2}^{l-1}}+1)}}$ and analogously every  ${{\beta }_{l({{2}^{l-1},y})}}$ and ${{\beta }_{l(x,y)}}$. Hence we can express from elements of $G^2$ every even number of active states on $X^l$.
\end{proof}



Define the subgroup $G(l)<Aut{{X}^{[k]}}$, where $l\le k$, as $Stab_{Aut{{X}^{[k]}}}(l)\left| _{{{X}^{l}}} \right.$.
Let us construct a homomorphism from $G(l)$ to ${{C}_{2}}$ in the following way: $\varphi (\alpha )=\sum\limits_{i=1}^{{{2}^{l}}}{{{s}_{li}}(\alpha )}\bmod 2$. Note that $\varphi (\alpha \cdot \beta )=\varphi (\alpha )\circ \varphi (\beta )=(\sum\limits_{i=1}^{{{2}^{l}}}{{{s}_{li}}(\alpha )}+\sum\limits_{i=1}^{{{2}^{l}}}{{{s}_{li}}(\beta )})mod2$.

Structure of subgroup $G_{k}^{2}{{G}_{k}}'\triangleleft \underset{1}{\overset{k}{\mathop{\wr }}}\,{{S}_{2}}\simeq Aut{{X}^{[k]}}$ can be described in next way. This subgroup contains the commutant -- ${G}_{k}'$ so it has on each ${{X}^{l}},\,\,\,0\le l<k-1$ all even indexes that can exists there. On ${{X}^{k-1}}$ it does not exist v.p. of type \texttt{T}, which has the distance $2k-2$, rest of even the indexes are present on ${X}^{k-1}$. It's so, because the sets of elements of types \texttt{T} and \texttt{C} are not closed under operation of calculating the even power as it proved in Lemma \ref{About not closed set of element of type T}.

Thus, the squares of the elements don't belong to \texttt{T} and \texttt{C} (because they have the distance, which is less than $2k-2$).
This implies the following corollary.

\begin{corollary} \label{qoutient} A quotient group ${}^{{{G}_{k}}}/{}_{G_{k}^{2}{{G}^{'}_{k}}}$ is isomorphic to $\underbrace{{{C}_{2}}\times {{C}_{2}}\times ...\times {{C}_{2}}}_{k}$.
\end{corollary}

\begin{proof}
The proof is based on two facts $G_{k}^{2}{{G}^{'}_{k}}\simeq G_{k}^{2}\triangleleft {{G}_{k}}$ and $\left| G:G_{k}^{2}{{G}^{'}_{k}} \right|=2^k$.
Construct a homomorphism from $G_k(l)$ to ${{C}_{2}}$ in the following way: $\varphi (\alpha )=\sum\limits_{i=1}^{{{2}^{l}}}{{{s}_{li}}(\alpha )}\bmod 2$. Note that $\varphi (\alpha \cdot \beta )=\varphi (\alpha )\circ \varphi (\beta )=(\sum\limits_{i=1}^{{{2}^{l}}}{{{s}_{li}}(\alpha )}+\sum\limits_{i=1}^{{{2}^{l}}}{{{s}_{li}}(\beta )})mod2$, where $\alpha ,\,\,\beta \in Aut{{X}^{[n]}}$.
Index of $\alpha \in G_{k}^{2}$ on ${{X}^{l}},\,l<k-1$ is even but index of  $\beta \in {{G}_{k}}$ on ${{X}^{l}}$ can be both even and odd. Note that ${{G}_{k}}(l)$ is abelian group and $G_{k}^{2}(l)\trianglelefteq {{G}_{k}}$.
  Since words with equal logarithms to all bases \cite{K} belong to distinct cosets of the commutator, the subgroup $G_{k}^{2}(l)$ is the kernel of this mapping. Also we can use homomorphism $\varphi $  which is described above and denote it as ${{\varphi }_{l}}$, to map ${{G}_{k}}(l)$ to ${{C}_{2}}$ the $\ker {{\varphi }_{l}}=G_{k}^{2}(l)$. Really if $\alpha $ from ${{G}_{k}}(l)$ has odd number of active states on ${{X}^{l}},\,\,\,l<k-1$ than  ${{\varphi }_{i}}(\alpha )=1$ in ${{C}_{2}}$ otherwise if this number is even than $\alpha $ from $\ker {{\varphi }_{i}}$ so ${{\varphi }_{i}}(\alpha )=0$ hence $\ker {{\varphi }_{i}}=G_{k}^{2}(l)$.
   So ${}^{{{G}_{k}}(l)}/{}_{G_{k}^{2}(l)}={{C}_{2}}$ analogously ${}^{{{B}_{k}}(l)}/{}_{B_{k}^{2}(l)}={{C}_{2}}$. Let us  check that mapping  $({{\varphi }_{0}},{{\varphi }_{1}},...,{{\varphi }_{k-2}},{{\phi }_{k-1}})$ is the homomorphism from ${{G}_{k}}$ to $\underbrace{{{C}_{2}}\times {{C}_{2}}\times ...\times {{C}_{2}}}_{k-1}$.
So we can construct and homomorphism ${{\varphi }_{i}}$ from every factor ${}^{{{G}_{k}}(i)}/{}_{G_{k}^{2}(i)}$ of this direct product to ${{C}_{2}}$.    The group ${}^{{{G}_{k}}}/{}_{G_{k}^{2}}$  is elementary abelian 2-group because ${{g}^{2}}=e,\,\,g\in G$.
Index subgroup $\varphi (\alpha \cdot \beta )=\left( \varphi (\alpha )+\varphi (\beta ) \right)\bmod 2$ because multiplilcation $\alpha \cdot \beta $ in ${{G}_{k}}$ does not change a parity of  index of $\beta $, $\beta \in {{G}_{k}}$ on ${{X}^{l}}$.  Really action of element of active group $A=\underbrace{{{S}_{2}}\wr {{S}_{2}}\wr ...\wr {{S}_{2}}}_{l-1}$  from wreath power $(\underbrace{{{S}_{2}}\wr {{S}_{2}}\wr ...\wr {{S}_{2}}}_{l-1})\wr {{S}_{2}}$ on element from passive subgroup ${{S}_{2}}$ of second multiplier from product $gf,\,\,g,f\in (\underbrace{{{S}_{2}}\wr {{S}_{2}}\wr ...\wr {{S}_{2}}}_{l-1})\wr {{S}_{2}}$ does not change a parity of index of $\beta $ on ${{X}^{l}}$, if index of $\beta $ was even then under action  it stands to be even and the sum $\varphi (\alpha )\bmod 2+\varphi (\beta )\bmod 2$ will be equal to $(\varphi (\alpha )+\varphi (\beta ))\bmod 2$,  hence it does not change a $\varphi (\beta )$. Since words with equal logarithms to all bases \cite{K} belong to distinct cosets of the commutator, the subgroup $G_{k}^{2}(l)$ is the kernel of this mapping.
Let us define the permutations of the type 2 that act on ${{X}_{1}}$ and ${{X}_{2}}$, where ${{X}_{1}}=\{{{v}_{k,1}},...,{{v}_{k{{,2}^{k-1}}}}\},\,\,{{X}_{2}}=\{{{v}_{k{{,2}^{k-1}}+1}},...,{{v}_{k{{,2}^{k}}}}\},\,{{X}_{1}}\cup {{X}_{2}}={{X}^{k}}$  only by even permutations. Subgroup $G_{k}^{2}G{{'}_{k}}$ acts only by permutations of type 2 on ${{X}_{1}}$, ${{X}_{2}}$, according to Statement 2.
The restriction ${{\left. G_{k}^{2} \right|}_{{{X}^{[k-1]}}}}$ acts only by permutations of the second type (elements of it form a normal subgroup in ${{G}_{k}}$) by parity of permutation on sets  ${{X}_{1}}$ and ${{X}_{2}}$.  A permutation of Type 1, where on ${{X}_{1}}$ and ${{X}_{2}}$  the group ${{G}_{k}}$  can acts by odd as well as by even permutations but in such way to resulting permutation on ${{X}^{k}}$ is always even. The number of active states from subgroup ${{G}_{k}}(k-1)$ on ${{X}^{k-1}}$ can be even as well as odd.
It means that on set of vertices of ${{X}^{k-1}}$  over ${{X}_{1}}$ i.e. vertices which are connected by edges with vertices of ${{X}^{k-1}}$  over ${{X}_{1}}$ automorphism of ${{G}_{k}}$ can contains odd number of active states (and ${{X}_{2}}$ analogously).
For subgroup ${{G}_{k}}(k-1)$  it was constructed a homomorphism onto ${{C}_{2}}$ as a sum of non-trivial states by $\,\bmod \,2$  on both sets ${{X}_{1}}$ and ${{X}_{2}}$.

 For subgroup $G_k(k-1)$  we construct a homomorphism onto ${C}_{2}$ as a sum of non-trivial v.p. by $\bmod 2$ on both sets ${X}_{1}$ and ${X}_{2}$.
For a subgroup $G(k-1)$ such that has the normal subgroup ${{G}_{k}}^2({k-1})\triangleleft {{G}_{k}}({k-1})$ we construct a homomorphism: $ {{\phi }_{k-1}} \left( {{G}_{k}}({{X}_{k}}) \right)\to {{C}_{2}} \simeq {}^{{G}_{k}}({k-1})/{}_{ G^2_k (k-1)} $ as product of sum by $mod2$ of active states (${{s}_{k-1,i}}\in \{0,1\}$, $0<j\le {{2}^{k-2}}$ if ${{s}_{ij}}\in {{X}_{1}}$)
   on each set ${{X}_{1}}$ and ${{X}_{2}}$:  $\phi_{\alpha} ({{X}_{1}})\,\, \cdot \,\,\phi_{\alpha} ({{X}_{2}})=\sum\limits_{i=1}^{{{2}^{k-2}}}{{{s}_{k-1,i}}(\alpha)}(\bmod 2) \cdot \sum\limits_{i={{2}^{k-2}}+1}^{{{2}^{k-1}}}{{{s}_{k-1,i}}(\alpha)} (\bmod 2)$.
  Where ${{s}_{k-1,i}}(\alpha)=1$ if  there is active state  in ${{v}_{k-1,i}}, \, i<2^{k-1}+1$ and ${{s}_{k-1,i}}(\alpha)=0$ if there is no active state. It follows from structure of ${{G}_{k}}$ that $\phi_{\alpha} ({{X}_{1}})\,\,=\,\,\phi_{\alpha} ({{X}_{2}})$ so it is 0 or 1. But $G_{k}^{2}{{G}^{'}_{k}}$ admits only permutations of Type 2 on ${{X}^{k}}$ so $G_{k}^{2}{{G}_{k}}'({{X}_{k}})\triangleleft {{G}_{k}}({{X}_{k}})$ because it holds a conjugacy and it is kernel of map from ${{G}_{k}}({{X}_{k}})$ onto ${{C}_{2}}$.

Hence for a subgroup $G(k-1)$ such that has the normal subgroup ${{G}_{k}}^{2}(k-1)\triangleleft {{G}_{k}}(k-1)$ it was constructed a homomorphism:  ${{\phi }_{k-1}}\left( {{G}_{k}}(k-1) \right)\to {{C}_{2}}\simeq {{G}_{k}}(k-1){{/}_{G_{k}^{2}(k-1)}}$ as product of sum by $mod2$ of active states from $X_1$ and $X_2$.
 As the result we have ${}^{{{G}_{k}}}/{}_{G_{k}^{2}}\simeq \underbrace{{{C}_{2}}\times {{C}_{2}}\times ...\times {{C}_{2}}}_{k-1}$.
\end{proof}
Considering that it was proved in Theorem 1 and Theorem 2 that ${{G}_{k}}\simeq A_{2^k}$ we can formulate next Corollary.
   \begin{corollary} The group $A_{2^k}$ has minimal system of generators from $k$ elements.
\end{corollary}
   \begin{proof}
Since quotient group of ${{G}_{k}}$ by subgroup of Frattini $G_{k}^{2}{{G}^{'}_{k}}$  has minimal system of generators from $k$ elements because ${}^{{{G}_{k}}}/{}_{G_{k}^{2}{{G}^{'}_{k}}}$  is isomorphic to linear $p$-space $(p=2)$ of dimension $k$ (or elementary abelian group) then ${{G}_{k}}$ has rank $k$  \cite{Rot}.
   \end{proof}
 \begin{main_theorem}
The set $S_{\mathop{\beta}}=\{\mathop{\beta}_{0}, \mathop{\beta}_{1}, \mathop{\beta}_{2}, \ldots , \mathop{\beta}_{k-2}, \tau \}$, where $\mathop{\beta}_{i} = \alpha_i$, 
 is a minimal generating system for a group $G_k$ that is isomorphic to Sylow 2-subgroup of $A_{2^{k}}$.
\end{main_theorem}
We have isomorphism of $G_k$ and $Syl_2 (A_{2^k})$ from Theorem \ref{isomor}, the minimality of $S_{\mathop{\beta}}$ following from
Lemma \ref{Lem About product elem of C} which said that $S_{\mathop{\beta}}$ has to contain an element of type \texttt{T}, Theorem 3 and Lemma \ref{rk} about minimal rank.
Another way to prove the minimality of $S_{\mathop{\beta}}$ is given to us by Corollary \ref{qoutient} about quotient by Frattini subgroup.

For example a minimal system of generators for $Syl_2(A_{8})$ can be constructed by following way, for convenience let us consider the next system:

\shorthandoff{"}
\def\vertex{\scriptscriptstyle\cdot}
\def\objectstyle{\scriptstyle}
\xy <1cm,0cm>:
(10,3.5)="C0"*{\ }+<-1.2ex,-0.45ex>*{ \beta_0  },
(11,3)="C1"*{\vertex}+<-1.2ex,-0.45ex>*{},
(10.5,2)="C5"*{\vertex}+<-1.2ex,-0.45ex>*{},
(11.5,2)="C6"*{\vertex}+<-1.2ex,-0.45ex>*{},
(10.33,1)="C7"*{\vertex}+<-1.2ex,-0.45ex>*{}, 
(10.66,1)="C8"*{\vertex}+<-1.2ex,-0.45ex>*{},
(11.33,1)="C9"*{\vertex}+<-1.2ex,-0.45ex>*{}, 
(11.66,1)="C10"*{\vertex}+<-1.2ex,-0.45ex>*{},
(13,3)="C2"*{\vertex}+<1.2ex,-0.45ex>*{},
(12.5,2)="C11"*{\vertex}+<-1.2ex,-0.45ex>*{},
(13.5,2)="C12"*{\vertex}+<-1.2ex,-0.45ex>*{},
(12.33,1)="C13"*{\vertex}+<-1.2ex,-0.45ex>*{},
(12.66,1)="C14"*{\vertex}+<-1.2ex,-0.45ex>*{},
(13.33,1)="C15"*{\vertex}+<-1.2ex,-0.45ex>*{},  
(13.66,1)="C16"*{\vertex}+<-1.2ex,-0.45ex>*{},  
(12,4)="L0"*{\bullet}+<0ex,1.5ex>*{1},
"C2";"L0"**@{-},
"C1";"C5"**@{-},
"C1";"C6"**@{-},
"C5";"C7"**@{--},  
"C5";"C8"**@{--},
"L0";"C1"**@{-},
"C6";"C9"**@{--},
"C6";"C10"**@{--},
"C2";"C11"**@{-},
"C2";"C12"**@{-},
"C11";"C13"**@{--},
"C11";"C14"**@{--},
"C12";"C15"**@{--},
"C12";"C16"**@{--},
(0,3.5)="A0"*{\ }+<-1.2ex,-0.45ex>*{ \tau},
(1,3)="A1"*{\vertex}+<-1.2ex,-0.45ex>*{},
(0.5,2)="A5"*{\bullet}+<-1.2ex,-0.45ex>*{1},
(1.5,2)="A6"*{\vertex}+<-1.2ex,-0.45ex>*{},
(0.33,1)="A7"*{\vertex}+<-1.2ex,-0.45ex>*{},
(0.66,1)="A8"*{\vertex}+<-1.2ex,-0.45ex>*{},
(1.33,1)="A9"*{\vertex}+<-1.2ex,-0.45ex>*{},
(1.66,1)="A10"*{\vertex}+<-1.2ex,-0.45ex>*{},
(3,3)="A2"*{\vertex}+<1.2ex,-0.45ex>*{},
(2.5,2)="A11"*{\vertex}+<-1.2ex,-0.45ex>*{},
(3.5,2)="A12"*{\bullet}+<-1.2ex,-0.45ex>*{1},
(2.33,1)="A13"*{\vertex}+<-1.2ex,-0.45ex>*{},
(2.66,1)="A14"*{\vertex}+<-1.2ex,-0.45ex>*{},
(3.33,1)="A15"*{\vertex}+<-1.2ex,-0.45ex>*{},
(3.66,1)="A16"*{\vertex}+<-1.2ex,-0.45ex>*{},
(2,4)="A4"*{\vertex}+<0ex,1.5ex>*{},
"A2";"A4"**@{-},
"A1";"A5"**@{-},
"A1";"A6"**@{-},
"A5";"A7"**@{--},
"A5";"A8"**@{--},
"A4";"A1"**@{-},
"A6";"A9"**@{--},
"A6";"A10"**@{--},
"A2";"A11"**@{-},
"A2";"A12"**@{-},
"A11";"A13"**@{--},
"A11";"A14"**@{--},
"A12";"A15"**@{--},
"A12";"A16"**@{--},
(5,3.5)="B0"*{\ }+<-1.2ex,-0.45ex>*{ \beta_1 },
(6,3)="B1"*{\bullet}+<-1.2ex,-0.45ex>*{1}, 
(5.5,2)="B5"*{\vertex}+<-1.2ex,-0.45ex>*{},
(6.5,2)="B6"*{\vertex}+<-1.2ex,-0.45ex>*{},
(5.33,1)="B7"*{\vertex}+<-1.2ex,-0.45ex>*{},
(5.66,1)="B8"*{\vertex}+<-1.2ex,-0.45ex>*{},
(6.33,1)="B9"*{\vertex}+<-1.2ex,-0.45ex>*{},
(6.66,1)="B10"*{\vertex}+<-1.2ex,-0.45ex>*{},
(8,3)="B2"*{\vertex}+<1.2ex,-0.45ex>*{},
(7.5,2)="B11"*{\vertex}+<-1.2ex,-0.45ex>*{},
(8.5,2)="B12"*{\vertex}+<-1.2ex,-0.45ex>*{}, 
(7.33,1)="B13"*{\vertex}+<-1.2ex,-0.45ex>*{},
(7.66,1)="B14"*{\vertex}+<-1.2ex,-0.45ex>*{},
(8.33,1)="B15"*{\vertex}+<-1.2ex,-0.45ex>*{},
(8.66,1)="B16"*{\vertex}+<-1.2ex,-0.45ex>*{},
(7,4)="B4"*{\vertex}+<0ex,1.5ex>*{},
"B2";"B4"**@{-},
"B1";"B5"**@{-},
"B1";"B6"**@{-},
"B5";"B7"**@{--},
"B5";"B8"**@{--},
"B4";"B1"**@{-},
"B6";"B9"**@{--},
"B6";"B10"**@{--},
"B2";"B11"**@{-},
"B2";"B12"**@{-},
"B11";"B13"**@{--},
"B11";"B14"**@{--},
"B12";"B15"**@{--},
"B12";"B16"**@{--},
\endxy

 Consequently, in such way we construct second generating system for $A_{2^k}$ of $k$ elements that is less than in \cite{Iv}, and this system is minimal.

\section{Structure and propeties of $Syl_2{A_n}$}

The structure of $Syl_2A_{12}$ is the same as of the subgroup $H_{12} < Syl_2(S_8) \times Syl_2(S_4)$, for which $[Syl_2(S_8) \times Syl_2(S_4):H_{12}]=2$. $|Syl_2(A_{12})|= 2^{[12/2] + [12/4]+ [12/8]-1} = 2^9$. Also $|Syl_2(S_8)|=2^7$, $|Syl_2(S_4)|=2^3$, so $|Syl_2(S_8) \times Syl_2(S_4)|=2^{10}$ and $|H_{12}|=2^9$, because its index in $Syl_2(S_8) \times Syl_2(S_4)$ is 2. The structure of $Syl_2(A_6)$ is the same as of $H_6 < Syl_2(S_4) \times (C_2)$. Here $H_6 = \{(g,h_g)|g \in Syl_2(S_4), h_g \in C_2\}$, where
\begin{equation}\label{H}
 \begin{cases}
h_g = e, \ \ if \ g|_{L_2} \in Syl_2(A_6), \\
h_g = (5,6), \ if \, g|_{L_2} \in {Syl_2(S_6) \setminus Syl_2(A_6)},
 \end{cases}
\end{equation}
The structure of $Syl_2(A_{6})$ is the same as subgroup $H_6:$ $H_6 < Syl_2(S_4) \times (C_2)$ where $H_6= \{ (g, h) | g\in Syl_2(S_4), h \in  AutX \}$. So last bijection determined by (\ref{H}) giving us $Syl_2 A_{6} \simeq Syl_2 S_{4} $. As a corollary we have $Syl_2 A_{{2^k}+2} \simeq Syl_2 S_{2^k} $.
The structure of $Syl_2(A_{7})$ is the same as of the subgroup $H_7:$ $H_7 < Syl_2(S_4) \times S_2$ where $H_6= \{ (g, h) | g\in Syl_2(S_4), h \in  S_2 \}$ and $h$ depends of $g$:
\begin{equation}\label{HH}
 \begin{cases}
h_g = e, \ \ if \  g|_{L_2}\in  Syl_2  A_7, \\
h_g = (i,j), i,j \in\{ 5,6,7 \},  \ if \, g|_{L_2}\in  {Syl_2 S_7\setminus Syl_2A_7}.
 \end{cases}
\end{equation}
The generators of the group $H_7$ have the form $(g,h), \, \, g\in Syl_2(S_4), \, h\in C_2$, namely: $ \{ {\beta_{0}; \beta_{1}, \tau} \} \cup \{ (5,6) \}$. An element $h_g$ can't be a product of two transpositions of the set: ${(i,j), (j,k), (i,k)}$, where $i,j,k$  $\in\{ 5,6,7 \} $, because $(i,j)(j,k)=(i,k,j)$ but $ord(i,k,j) =3$, so such element doesn't belong to 2-subgroup. In general elements of $Syl_2 A_{4k+3}$ have the structure (\ref{HH}), where $h_g = (i,j), \,\, i,j \in\{ 4k+1, 4k+2, 4k+3 \}$ and $g\in Syl_2 S_{4k}$.

Also $|Syl_2(S_8)|=2^7$, $|Syl_2(S_4)|=2^3$, so $|Syl_2(S_8) \times Syl_2(S_4)|=2^{10}$ and $|H_{12}|=2^9$, because its index in $Syl_2(S_8) \times Syl_2(S_4)$ is 2. The structure of $Syl_2(A_6)$ is the same as of $H_6 < Syl_2(S_4) \times (C_2)$. Here $H_6 = \{(g,h_g)|g \in Syl_2(S_4), h_g \in C_2\}$.

The orders of this groups are equal, really $|Syl_2(A_7)|= 2^{[7/2] + [7/4]-1}  = 2^3= |H_7|$. In case \, $g|_{L_2}\in  {S_7\setminus A_7}$ we have $C_3^2$ ways to construct one transposition that is direct factor in $H$ which complete $Syl_2 S_4$ to $H_7$ by one transposition  $: \{(5,6); (6,7); (5,7) \}$.

The structure of $Syl_2(A_{2^k+2^l})$ $(k>l)$ is the same as of the subgroup $H_{2^k+2^l} < Syl_2(S_{2^k}) \times Syl_2(S_{2^l})$, for which $[Syl_2(S_{2^k}) \times Syl_2(S_{2^l}):H]=2$. $|Syl_2(A_{2^k+2^l})|= 2^{[(2^k+2^l)!/2] + [(2^k+2^l)!/4]+ .... -1} $.
Here $H = \{(g,h_g)|g \in Syl_2(S_2^k), h_g \in Syl_2(S_2^l\}$, where

\begin{equation}\label{HHH}
 \begin{cases}
h  \in    A_{2^l}, \ \ if \  g|_{X^{k-1}}\in    A_{2^k}, \\
h:  h|_{X^2}\in  {Syl_2 (S_{2^l}) \setminus Syl_2 (A_{2^l})},  \ if \, g|_{X^k}\in  {Syl_2 (S_{2^k}) \setminus Syl_2 (A_{2^k})}.
 \end{cases}
\end{equation}
The generators of the group $H_7$ have the form $(g,h), \, \, g\in Syl_2(S_4), \, h\in C_2$, namely: ${\beta_{0}; \beta_{1}, \tau} \cup {(5,6)}$.

I.e. for element  ${{\beta }_{\sigma }}(2i-1)=2\sigma (i)-1,\,\,{{\beta }_{\sigma }}(2i)=2\sigma (i)$, ${{\sigma }_{i}}\in \left\{ {{1,2,...,2}^{k-1}} \right\}$.

\begin{property} Relation between structures of Sylow subgroups is given by $ Sy{{l}_{2}}({{A}_{4k+2}}) \simeq Sy{{l}_{2}}({S_{4k}})$, where $k\in \mathbb{N}$.
\end{property}
\begin{proof}
Here $H_6 = \{(g,h_g)|g \in Syl_2(S_4), h_g \in C_2\}$, where

\begin{equation}\label{HHHH}
 \begin{cases}
h_g = e, \ \ if \ g|_{L_k} \in Syl_2(A_{4k+2}), \\
h_g = (4k+1,4k+2), \ if \, g|_{L_k} \in Syl_2(S_{4k+2}) \setminus Syl_2(A_{4k+2}),
 \end{cases}
\end{equation}

The structure of $Syl_2(A_{6})$ is the same as subgroup $H_6:$ $H_6 < Syl_2(S_4) \times (C_2)$ where $H_6= \{ (g, h) | g\in Syl_2(S_4), h \in  AutX \}$. So last bijection determined by (\ref{HHHH}) give us $Syl_2 A_{6} \simeq Syl_2 S_{4} $. As a corollary we have $Syl_2 A_{{2^k}+2} \simeq Syl_2 S_{2^k} $.
\end{proof}
\begin{property} Relation between orders of Sylows subgroup for $n=4k-2$ and $n=4k$ is given by $\left| Sy{{l}_{2}}({{A}_{4k-2}}) \right|={{2}^{i}}\left| Sy{{l}_{2}}({{A}_{4k}}) \right|$, where value $i$ depend only of power of 2 in decomposition of prime number of $k$.
\end{property}
\begin{proof}
Really $\left| {{A}_{4k-2}} \right|=\frac{(4k-2)!}{2}$, therefore $\left| {{A}_{4k}} \right|=\frac{(4k-2)!}{2}(4k-1)4k$, it means that $i$ determines only by $k$ and isn't bounded.
 \end{proof}
\begin{lemma} \label{islmorph} If $n=4k+2$, then the group $Syl_2(A_n)$ is isomorphic
to $Syl_2(S_{4k})$.
\end{lemma}

\begin {proof} Bijection correspondence between set of elements of $Syl_2(A_n)$ and $Syl_2(S_{4k})$ we have from (\ref{HHHH}). Let's consider a mapping $\phi: Syl_2 (S_{4k}) \rightarrow Syl_2 (A_{4k+2})$ if $\sigma \in Syl_2(S_{4k})$ then $\phi(\sigma)=\sigma \circ (4k+1, 4k+2)^{\chi(\sigma)}=(\sigma,  (4k+1, 4k+2)^{\chi(\sigma)})$, where $\chi(\sigma)$ is number of transposition in $\sigma$ by module 2.

So $\phi(\sigma) \in Syl_2(A_{4k+2})$.
 If $\phi(\sigma) \in A_{n}$ then ${\chi(\sigma)}=0$, so $\phi(\sigma) \in Syl_2(A_{n-1})$. Check that $\phi$ is homomorphism.
Assume that ${{\sigma }_{1}}\in Sy{{l}_{2}}({{S}_{4k}}\backslash {{A}_{4k}}),\,\,{{\sigma }_{2}}\in Sy{{l}_{2}}({{A}_{4k}})$, then $\phi ({{\sigma }_{1}})\phi ({{\sigma }_{2}})=({{\sigma }_{1}},{{h}^{\chi({{\sigma }_{1}})}})({{\sigma }_{2}},e)=({{\sigma }_{1}}{{\sigma }_{2}},h)={{\sigma }_{1}}{{\sigma }_{2}}\circ (4k+1,4k+2)$, where $({{\sigma }_{i}},h)={{\sigma }_{i}}\circ {{h}^{\chi({{\sigma }_{i}})}}\in Sy{{l}_{2}}({{A}_{4k+2}})$. If ${\sigma _{1}},\,\,{\sigma_{2}}\in {{S}_{{{2}^{k}}}}\backslash {{A}_{{{2}^{k}}}}$, then $\phi ({{\sigma }_{1}})\phi ({{\sigma }_{2}})=({{\sigma }_{1}},{{h}^{\chi({{\sigma }_{1}})}})({{\sigma }_{2}},{{h}^{\chi({{\sigma }_{2}})}})=({{\sigma }_{1}}{{\sigma }_{2}},\,e)=(a,\,e)$, where ${{\sigma }_{1}}{{\sigma }_{2}}=a\in {{A}_{4k+2}}$.
So it is isomorphism.
\end {proof}

\begin{remark}
If $n=4k$, then index $Syl_2(A_{n+3})$ in $A_{n+3}$ is equal to $[S_{4k+1}: Syl_2 (A_{4k+1})](2k+1)(4k+3)$, index
$Syl_2(A_{n+1})$ in $A_{n+1}$ as a subgroup of index $2^{m-1}$, where $m$ is the
maximal natural number, for which $4k!$ is divisible by $2^m$.
\end {remark}
 \begin {proof} For $Syl_2(A_{n+3})$ its order equal to maximal power of 2 which divide $(4k+3)!$ this power on 1 grater then correspondent power in $(4k+1)!$  because $(4k+3)!=(4k+1)!(4k+2)(4k+3)=(4k+1)!2(2k+1)(4k+3)$ so $\mid Syl_2 A_{n+3}\mid= 2^m \cdot 2 = 2^{m+1}$.
As a result of it indexes of $A_{n+3}$ and $A_{n+1}$ are following: $  [S_{4k+1}: Syl_2 (A_{4k+1}) ] = \frac{(4k+1)!}{2^m} $ and $[S_{4k+3}: Syl_2 (A_{4k+3})] = [S_{4k+1}: Syl_2 (A_{4k+1}) ](2k+1)(4k+3) = \frac{(4k+1)!}{2^m}(2k+1)(4k+3) $.
\end {proof}

\begin{remark}
If $n=2k$ then
$[Syl_2(A_n) : Syl_2(S_{n-1})] =2^{m-1}$, where $m$ is the maximal power of 2 in factorization of $n$.
\end {remark}

 \begin {proof}
$|Syl_2(S_{n-1})|$ is equal to $t$ that is a maximal power of 2 in $(n-1)!$. $|Syl_2(A_{n})|$ is equal to maximal power of 2 in $(n!/2)$. Since $n=2k$ then $(n/2)!=(n-1)!\frac{n}{2}$ and $2^f$ is equal to product maximal power of 2 in $(n-1)!$ on maximal power of 2 in $\frac{n}{2}$. Therefore $\frac{|Syl_2(A_{n})|}{|Syl_2(S_{n-1})|}=\frac {2^{m-1}}{2^t} 2^t=2^{m-1}. $
Note that for odd $m=n-1$ the group $Syl_2(S_{m}) \simeq Syl_2(S_{m-1})$ i.e. $Syl_2(S_{n-1})\simeq Syl_2(S_{n-2})$. The group $Syl_2(S_{n-2})$ contains the automorphism of correspondent binary subtree with last level $X^{n-2}$ and this automorphism realizes the permutation $\sigma$ on $X^{n-2}$. For every $\sigma\in Syl_2(S_{n-2})$ let us set in correspondence a permutation $\sigma (n-1,n)^{\chi (\sigma)} \in Syl_2(A_{n})$, where $\chi (\sigma)$ -- number of transposition in $\sigma$ by $mod\, 2$, so it is bijection $\phi(\sigma)\longmapsto \sigma (n-1, n){\chi (\sigma)}$ which has property of homomorphism, see Lemma \ref{islmorph}. Thus we prove that $Syl_2(S_{n-1}) \hookrightarrow Syl_2(A_{n})$ and its index is $2^{d-1}$.
\end {proof}

\begin{remark}
The ratio of $|Syl_2(A_{4k+3})|$ and $|Syl_2(A_{4k+1})|$ is equal to 2 and ratio of indexes $[A_{4k+3} : Syl_2(A_{4k+3})]$ and $[A_{4k+1} : Syl_2(A_{4k+1})]$ is equal $(2k+1)(4k+3)$.
\end{remark}

\begin{proof} The
ratio $|Syl_2(A_{4k+3})| : |Syl_2(A_{4k+1})|= 2$ holds because formula of Legendere gives us new one power of 2 in $(4k+3)!$ in compering with $(4k+1)!$.  Second part of statement follows from theorem about $p$-subgroup of $H$, $[G:H] \neq kp $ then one of $p$-subgroups of $H$ is Sylow $p$-group of $G$. In this case $p=2$ but $|Syl_2(A_{4k+3})| : |Syl_2(A_{4k+1})|=2$ so we have to divide ratio of indexes on $2$.
\end{proof}


Let us consider function of Morse \cite{Shar} $f:\,{D^2}\to \mathbb{R}$ that painted at pict. 2 and graph of Kronrod-Reeb \cite{Maks} which obtained by contraction every set's component  of level of ${{f}^{-1}}(c)$ in point. Group of automorphism of this graph is isomorphic to $Sy{{l}_{2}}{{S}_{{{2}^{k}}}}$ where $k=2$ in general case we have regular binary rooted tree for arbitrary $k\in \mathbb{N}$.


\begin{figure}[h]
\begin{minipage}[h]{0.49\linewidth}
\center{\includegraphics[width=0.9\linewidth]{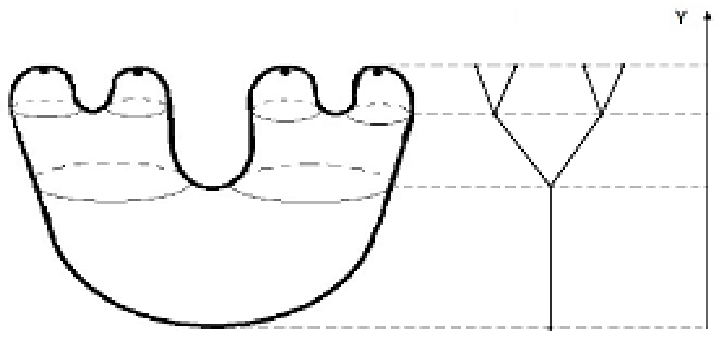} \\ Picture 2.}
\end{minipage}
\end{figure}
According to investigations of \cite{Maks2} for $D^2$ we have that $Syl_{2}S_{2^k} > G_k \simeq Syl_{2} A_{2^k}$ is quotient group of diffeomorphism group which stabilize function and isotopic to identity. Analogously to investigations of \cite{Maks, Maks2, SkThes} there is short exact sequence
$0\to {{\mathbb{Z}}^{m}}\to {{\pi }_{1}}{{O}_{f}}(f)\to G\to 0$, where $G$-group of automorphisms Reeb's (Kronrod-Reeb) graph \cite{Maks} and $O_f(f)$ is orbit under action of diffeomorphism group, so it could be way to transfer it for a group $Sy{{l}_{2}}(S_{2^k})$, where $m$ in ${{\mathbb{Z}}^{m}}$ is number of inner edges (vertices) in Reeb's graph, in case for $Syl_{2}S_{4}$ we have $m=3$.

Higher half of projection of manifold from pic. 2 can be determed by product of the quadratic forms $-({{(x+4)}^{2}}+{{y}^{2}})({{(x+3)}^{2}}+{{y}^{2}})({{(x-3)}^{2}}+{{y}^{2}})({{(x-4)}^{2}}+{{y}^{2}})=z$ in points $(-4,0) (-3,0) (3,0) (4,0) $ it reach a maximum value 0. Generally there is $-d_{1}^{2}d_{2}^{2}d_{3}^{2}d_{4}^{2}=z$.

Also from Statement \ref{comm} and corollary from it about $(AutX^{[k]})'$ can be deduced that derived length of $Syl_2 A_2^k$ is not always equal to $k$ as it was said in Lemma 3 of \cite{Dm} because in case $A_{2^k}$ if $k=2$ its $Syl_2 A_4 \simeq K_4$ but $K_4$ is abelian group so its derived length is 1.

\section{ Conclusion }
The proof of minimality of constructed generating set was done, also the description of the structure $Syl_2 A_{2^k}$, $Syl_2 A_{n}$ and its property was founded.

\end{document}